\documentclass[11pt]{article}

\usepackage[height=22cm,width=16cm,a4paper]{geometry}
\usepackage[dvipsnames]{xcolor}

\newcommand{\bsm}{\left(\begin{smallmatrix}} 
\newcommand{\esm}{\end{smallmatrix}\right)}  
\newcommand {\eh}{{\textstyle \frac{1}{2}}}
\usepackage[final]{hyperref}
\hypersetup{colorlinks=true,citecolor=blue}
\usepackage{graphicx}

\usepackage[latin1]{inputenc}
\usepackage{amssymb, amsfonts, amsmath, amsthm}
\usepackage{graphicx}
\usepackage{caption}
\usepackage{subcaption}

\newcommand{\norm}[1]{\Vert #1 \Vert}

\newcommand{\sprod}[2]{\langle #1, #2 \rangle}

\newcommand{\dds}{\frac{d}{ds}}
\newcommand{\ddt}{\frac{d}{dt}}
\newcommand{\dd}[2]{\frac{\partial #1}{\partial #2}}

\newcommand{\intd}{\, \text{d}}

\newcommand{\sd}{\, | \,}

\renewcommand{\phi}{\varphi}
\renewcommand{\epsilon}{\varepsilon}

\newcommand{\bR}{\mathbb{R}}
\newcommand{\bN}{\mathbb{N}}
\newcommand{\bZ}{\mathbb{Z}}

\DeclareMathOperator{\im}{im}
\DeclareMathOperator{\supp}{supp}
\DeclareMathOperator{\conv}{conv}
\newcommand{\atr}{\alpha^+}
\newcommand{\Itr}{I^+}
\newtheorem{thm}{Theorem}[section]
\newtheorem{prop}[thm]{Proposition}
\newtheorem{lem}[thm]{Lemma}
\newtheorem{cor}[thm]{Corollary}
\newtheorem{defin}[thm]{Definition}
\newtheorem{rem}[thm]{Remark}
\newtheorem{ex}[thm]{Example}
\begin{document}

\title{Positive topological entropy for multi-bump magnetic fields}

\author{Andreas Knauf\\Department Mathematik\\Friedrich-Alexander-Universit\"at Erlangen-N\"urnberg \and Frank Schulz\\Fakult\"at f\"ur Mathematik\\Technische Universit\"at Dortmund \and Karl Friedrich Siburg\\Fakult\"at f\"ur Mathematik\\Technische Universit\"at Dortmund}

\maketitle

\abstract{We study the dynamics of a charged particle in a planar magnetic field which consists of $n\geq 2$ disjoint localized peaks. We show that, under mild geometric conditions, this system is semi-conjugated to the full shift on $n$ symbols and, hence, carries positive topological entropy.}


\section{Introduction and results}

Consider a magnetic field in $\bR^3$ whose field lines are perpendicular to the plane $\bR^2\times \{0\}\cong \bR^2$. Then the motion of a particle of unit mass and unit charge in that plane is modelled by Newton's Second Law
\begin{equation} \label{eq:newton}
	\ddot{q} = B(q)J\dot{q}
\end{equation}
where $B: \bR^2\to\bR$ describes the field strength and the term on the right hand side is the Lorentz force corresponding to the magnetic field, with $J$ being the symplectic matrix 
$\bsm 0 & 1\\ -1 & 0\esm$. 
The differential equation \eqref{eq:newton} can be written as the Hamiltonian system 
generated by the Hamiltonian 
$H: T^* \bR^2 \to \bR, H(q,p) = \frac{1}{2} \norm{p}^2$ on $(T^* \bR^2, \omega)$ 
with the twisted symplectic form $$ \omega = \omega_0 + B(q) dq_1 \wedge dq_2 $$ 
were $\omega_0 = d\lambda = dp_1 \wedge dq_1 + dp_2 \wedge dq_2$ stands for the 
standard symplectic form on $T^* \bR^2$.

In this paper, we study the dynamics of a particle when the magnetic consists of $n\geq2$ 
localized peaks, i.e., when the support of $B$ consists of $n$ connected components 
$\supp B_k$. Assuming that each component is a disc where the magnetic field is rotationally 
symmetric, we can show that this dynamical systems exhibits chaotic behavior in the sense 
that it possesses an invariant set on which the Poincar\'e map induced by its flow is 
semi-conjugated to the full shift in $n$ symbols. This implies that there are solutions visiting 
the different components $\supp B_k$ in any prescribed order; see 
Fig.~\ref{fig:PoincareIndication}.

Moreover, we can conclude that for all sufficiently small energies $E>0$ our system has topological entropy $h_{top} \ge c \sqrt{E} > 0$.

\begin{figure}[ht]%
\centering
\includegraphics[width=0.4\columnwidth]{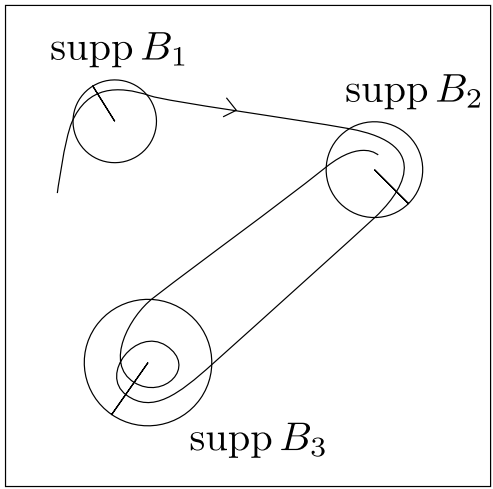}%
\caption{Symbolic Dynamics: Trajectory realising the prescribed order $..,1,2,3,3,2,..$}%
\label{fig:PoincareIndication}%
\end{figure}

There is a vast literature on similar results for the motion of a particle in a multi-bump
potential where the Hamiltonian is given by $H(q,p)=\frac{1}{2} \norm{p}^2 + V(q)$, 
and $\omega=\omega_0$. Classical scattering by potentials is treated in \cite{dg}. 
References on chaotic scattering can be found, \emph{e.g.}, in \cite{Sm,RR}.
Scattering by obstacles  is discussed in Sect.\ 5 of \cite{Ga}. 
Many of these works (including the case \cite{kn2} of molecular scattering) contain an
analysis of the bounded orbits using symbolic dynamics. In the abovementioned
cases the base for such an analysis is that scattering by a single center is described by a 
nontrivial topological index \cite{kn1,KK}.

Recent progress on motion in magnetic fields includes the following results.
\begin{itemize}
\item 
It is known that the horocycle flow on the upper half plane can be interpreted as
a magnetic geodesic flow. Dividing out co-compact subgroups of ${\rm PSL}(2,\bR)$, this gives rise
to magnetic geodesic flows on compact surfaces without periodic orbits.  
However both the metric and the magnetic field where shown in \cite{Pa} 
to be rigid in this case.
\item 
This was used in \cite{Gi} to provide counterexamples to the Hamiltonian 
Seifert conjecture for $\bR^{2d}$, $d\ge 3$ (that is, constructing 
a smooth proper Hamiltonian exhibiting a regular level set without periodic orbits).
\item 
In \cite{Gr} bounds for entropies of the magnetic geodesic flow on negatively curved manifolds were obtained, and topological entropy was shown to be strictly larger than the Liouville metric entropy except
in the constant curvature and zero field case. 

For magnetic geodesic flows on general closed surfaces criteria for positive topological entropy
were derived in  \cite{Mi}. Relationships between the growth rates of magnetic geodesics
and the topological entropy can be found in \cite{Ni}.
\item 
In \cite{PS} the magnetic geodesic flow was compared with the geodesic flow for
energies above the Ma\~n\'e critical value. Ma\~n\'e's action potential was found to coincide 
with the Riemannian length.
\end{itemize}

Surprisingly, the study of the dynamics of multi-bump magnetic fields has, to the best of our knowledge, not been undertaken yet. Note that, in contrast to the case of a singular potential, the magnetic Hamiltonian flow is complete. Indeed, since the energy $E = H$ is constant along the trajectories, we have $$ \norm{p(t,x_0)} = \sqrt{2E} $$ and 
$$ \norm{q(t,x_0)} \;=\; 
\norm{q(0,x_0) + \textstyle\int_0^t p(t,x_0)} \;\le \; \norm{q(0,x_0)} + |t| \sqrt{2E} $$ 
for all $t \in \bR$. 
Moreover, the energy surfaces $\Sigma_E := H^{-1}(E)$ with $E>0$ are all diffeomorphic to $\bR^2 \times S^1$, so the embedding of symbolic dynamics will not be based on the nontrivial topology of the energy surface but needs to be established by other methods.

The paper is organized as follows. In Section~\ref{Sect:Symm}, we will first study the dynamics generated by one component, i.e., the dynamics in a rotationally symmetric magnetic field. The main results will be then stated and proven in Section~\ref{Sect:SymbDyn}. Finally, in Section~\ref{Sect:Conclusion}, we make some further remarks and point out open problems we will attack in the future.

\emph{Acknowledgements:} The second author acknowledges financial support by the German National Academic Foundation.


\section{Dynamics in a rotationally symmetric magnetic field} \label{Sect:Symm}

First of all, it is necessary to understand the dynamics generated by one of the components. Therefore we consider the case of a compactly supported, rotationally symmetric magnetic field $B: \bR^2 \to \bR$. For the sake of simplicity, we will denote the induced function $\norm{q}\mapsto B(q)$ by the same letter $B$, and define the radius $R:=\sup\supp B>0$.

Due to the rotational symmetry of $B$, Noether's theorem implies the existence of a second constant of motion besides the kinetic energy. In order to find an explicit formula for it, we remark that \eqref{eq:newton} can be written as a Hamiltonian system in an alternative way. Since $\omega - \omega_0 = B(q) dq_1 \wedge dq_2 = B(r)r dr\wedge d\phi = d \alpha$ is exact with $\alpha = \big( -\int_r^R B(r)r \intd{r}\big) d\phi$, \eqref{eq:newton} is generated by the Hamiltonian $H(q,p) = \frac{1}{2} \norm{p - A(q)}^2$ with the magnetic vector potential $$ A(q) := \Big( \frac{1}{\norm{q}^2} \int_{\norm{q}}^R B(r)r \intd{r} \Big) Jq , $$ but now taken with respect to the standard symplectic form $\omega_0$. The corresponding Lagrangian is given by $$ L(q,\dot q) = \eh \norm{\dot q}^2 + \sprod{A(q)}{\dot{q}}. $$

\begin{prop} \label{integralCartesianCoords}
The quantity $I := \sprod{\dot q + A(q)}{-Jq}$ is an integral of motion. In fact, the system \eqref{eq:newton} with a rotationally symmetric magnetic field $B=B(\|q\|)$ is completely integrable.
\end{prop}

We call $I$ the \emph{magnetic momentum}.

\begin{proof}
We apply Noether's theorem (see, {\em e.g.}, Abraham and Marsden \cite{am})
to the one-parameter family 
$$ h^s: \bR^2 \to \bR^2,\ h^s(q) = \begin{pmatrix} \cos s & -\sin s\\ \sin s & \cos s \end{pmatrix} q =: T_sq . $$
For a solution $q = q(t)$ we have $\ddt h^s(q) = T_s \dot q$ and hence 
$$\textstyle L\big(h^s(q), \ddt h^s(q)\big) = 
\eh \norm{T_s \dot q}^2 + \sprod{A(T_s q)}{T_s\dot q}. $$ 
Since our choice of the magnetic potential $A$ satisfies $A(T_s q) = T_s A(q)$, we obtain 
\[L\big(h^s(q), \textstyle\ddt h^s(q)\big) \; =\;  \eh \norm{\dot q}^2 + \sprod{A(q)}{\dot q} 
\; = \; L(q, \dot q).\] 
Now Noether's theorem yields that 
$$ I \; = \; \textstyle  \sprod{\dd{L}{\dot{q}}(q,\dot q)}{\left.\dds\right|_{s=0} h^s(q)} 
\; = \; \sprod{\dot q + A(q)}{-Jq} $$ is an integral of motion.

The energy is a second integral of motion. It is a straightforward calculation to show that both integrals are in involution, hence our 4-dimensional Hamiltonian system is completely integrable.
\end{proof}

Now consider a solution $q(t)$ of \eqref{eq:newton} entering and then leaving the support 
of the magnetic field $B$. At the unique entering and exiting points $q^\pm:= q(t^\pm)$ we have 
$\norm{q^{\pm}}=R$, hence $A(q^{\pm})=0$. 
In view of Prop.~\ref{integralCartesianCoords}, this yields the following result.

\begin{cor} \label{enteringExitingAngleEqual}
For an orbit entering and exiting $\supp B$, the entering and exiting angles 
$\arccos \big(\langle \dot q(t^{\pm}), Jq^{\pm}\rangle/(\sqrt{2E}R)\big)$ coincide.
\end{cor}

Note that in polar coordinates $q = (r \cos \theta, r \sin \theta)$ the Hamiltonian satisfies $H = \frac{1}{2}({\dot r}^2 + r^2{\dot\theta}^2)$. Therefore we obtain, for some fixed energy $E >0$, the expression
\begin{equation} \label{eq:integralPolarCoords}
I = r^2\dot\theta - \int_r^R B(r) r \intd{r} = \pm r\sqrt{2E-{\dot r}^2} - \int_r^R B(r) r \intd{r} ,
\end{equation}
showing that $I$ does not depend on $\theta$.

In a constant magnetic field the trajectories are circles of a fixed radius, the Larmor radius. 
The question arises whether there are still circular orbits in a rotationally symmetric magnetic 
field. The first observation is that the curvature of a solution $q$ of 
\eqref{eq:newton} at $t$ equals $\frac{-B(q(t))}{\sqrt{2E}}$. 
Hence, the existence of a circular orbit in $\Sigma_E= H^{-1}(E)$ of radius $r$ 
w.r.t.\ the origin of the configuration plane is equivalent 
to $r$ satisfying the equation 
$$ \frac{|B(r)|}{\sqrt{2E}} = \frac{1}{r} . $$ 
This yields the following result.

\begin{lem} \label{exCircOrbit}
For every energy $E \in (0, E^\circ)$ with 
$$ E^\circ := \max_{r \ge 0} \frac{\big(B(r) r\big)^2}{2} $$ 
there are at least two circular orbits; for $E = E^\circ$ there is at least one circular orbit.
\end{lem}

For the following considerations we fix an energy $E \in (0, E^\circ]$. If we denote by $R^+ = R^+(E)$ the largest radius of the circular orbits having energy $E$, then $$ \frac{|B(r)|}{\sqrt{2E}} \begin{cases} = \frac{1}{r} & \text{for } r = R^+ \\ < \frac{1}{r} & \text{for } r > R^+ \end{cases} . $$

This outermost circular orbit plays an important role for the dynamics. Outside the disc of radius $R^+$ the magnetic field is too weak to capture orbits and prevent them from escaping to infinity. This is expressed by the following result.

\begin{prop} \label{VBed-Allg}
Let $E \in (0, E^\circ]$ and $x_0 = (q_0,p_0) \in \Sigma_E$ with $\norm{q_0} > R^+$ and $\sprod{q_0}{p_0} \ge 0$. Then there is a $\delta > 0$ such that $$ \norm{q(t)}^2 \ge \norm{q_0}^2 + \delta t^2 $$ for all $t \ge 0$.
\end{prop}

\begin{proof}
On $\{ \norm{q} \ge \norm{q_0}\}$ the function $\frac{|B(q)|}{\sqrt{2E}}\norm{q} < 1$ attains its maximum with maximal value, say, $1 - \frac{\delta}{2E}$. We consider the function $f(t) := \frac{1}{2} \norm{q(t)}^2$ with derivatives
\[
f'(t) = \sprod{q(t)}{p(t)}
\]
and
\[
f''(t) \; =  \; 2E + \sprod{q(t)}{B\big(q(t)\big)Jp(t)}  \; \ge \;  2E - \norm{q(t)}\, |B(q(t))|\, \sqrt{2E}.
\]
As long as $\norm{q(t)} \ge \norm{q_0}$ the second derivative satisfies the inequality
\[
f''(t) \ge 2E - \sqrt{2E}\Big(1 - \frac{\delta}{2E}\Big)\sqrt{2E} = \delta.
\]
Since $f'(0) \ge 0$ this holds for any $t \ge 0$ and hence the claim follows.
\end{proof}

\begin{rem}
The calculation in the proof shows that $t\mapsto \norm{q(t)}^2$ is convex while 
$q(t)$ is outside the disc $\{ \norm{q} \le R^+ \}$. Hence, for an orbit staying outside the disc of radius $R^+$, $\norm{q(t)}$ either attains its minimum or converges to its infimum as $t \to \infty$.
\end{rem}

\begin{rem}
The sign of $B(R^+)\neq 0$ determines the orientation of the outermost circular orbit. The orbit winds clockwise around the origin if $B(R^+)>0$, and counterclockwise if $B(R^+)<0$.
\end{rem}

For further investigations of the dynamics, especially of the orbits staying outside the $R^+$-disc, we make use of the integral and define $\Itr$ to be the value of $I$ on the circular orbit of radius $R^+$. For now we assume $B(R^+) > 0$ and refer to Rem.~\ref{integralRThetaGeneral} for the case $B(R^+) < 0$. We then have
\begin{equation} \label{eq:DefTrMom-}
\Itr = -R^+ \sqrt{2E} - \int_{R^+}^R B(r) r \intd{r}
\end{equation}
and call this quantity the \emph{critical magnetic momentum}.

In order to understand the motion of orbits entering the support of $B$, we consider the set of points
\begin{equation} \label{eq:UE1}
 U_E = \big\{ (q,p) \in \Sigma_E \sd \norm{q} = R, \sprod{q}{p} \le 0 \big\} 
 \end{equation}
through which orbits in $q$-space enter the ball of radius $R$. By taking the cosine of the angle between $q$ and $Jp$ we have $U_E \cong S^1 \times [-1,1]$. Since $\sprod{q}{Jp} = r^2 \dot\theta$ we have a correspondence between $\dot\theta \in [- \frac{\sqrt{2E}}{R},\frac{\sqrt{2E}}{R}]$ and the entering direction in $[-1,1]$. Thus, orbits entering the support of $B$ have, in view of \eqref{eq:integralPolarCoords}, magnetic momentum $I = R^2 \dot\theta \in [-R \sqrt{2E}, R \sqrt{2E}]$.

\begin{lem} \label{qGeR}
If $x=(q,p) \in U_E$ has magnetic momentum $I\leq \Itr$ then 
$\norm{q(t,x)} \ge R^+$ for all $t \in \bR$.
\end{lem}

\begin{proof}
Outside $\supp B$ the motion coincides with the free motion, hence the statement $\norm{q(t,x)} \ge R \ge R^+$ holds for all $t \le 0$. Assume a trajectory intersects the circle of radius $R^+$. Then at the intersection point we have $r = R^+$ and, by uniqueness of the solution, $\dot r \neq 0$. This implies 
\begin{align*}
I &= (R^+)^2\dot \theta - \int_{R^+}^R B(r)r \intd{r} \ge -R^+ \sqrt{2E - {\dot r}^2} - \int_{R^+}^R B(r)r \intd{r} \\
 &> -R^+ \sqrt{2E} - \int_{R^+}^R B(r)r \intd{r} = \Itr,
\end{align*}
contradicting our assumption.
\end{proof}

\begin{lem}\label{thetaLeZero}
Let $x=(q,p) \in \Sigma_E$ with $\norm{q} \ge R^+$ and $I \le \Itr$. 
Then, for any $\rho \ge \norm{q}$, there is a constant $c_{\rho} < 0$ such that 
$\dot \theta \le c_{\rho}$ as long as $\norm{q(t,x)} \le \rho$.
\end{lem}

\begin{proof}
The inequality
\begin{equation*}
\begin{split}
\dot\theta &= \frac{1}{\norm{q(t)}^2} \Big(I(x) + \int_{\norm{q(t)}}^{\rho} B(r)r \intd{r}\Big)\\
 &\le \frac{1}{\norm{q(t)}^2} 
 \Big(-R^+\sqrt{2E} - \int_{R^+}^{\rho} B(r)r \intd{r} + \int_{\norm{q(t)}}^{\rho} B(r)r \intd{r}\Big)\\
 &= \frac{1}{\norm{q(t)}^2} \Big(-R^+\sqrt{2E} - \int_{R^+}^{\norm{q(t)}} B(r)r \intd{r}\Big)\\
 &\le - \frac{R^+\sqrt{2E}}{\norm{q(t)}^2}\\
 &\le - \frac{R^+\sqrt{2E}}{\rho^2} =: C_{\rho}
\end{split}
\end{equation*}
holds as long as $\norm{q(t)} \le \rho$.
\end{proof}

Note that, in view of Lemma~\ref{qGeR}, points with $I< \Itr$ cannot enter the disc of radius $R^+$. Furthermore, for these trajectories we already know that $\norm{q(t)}$ assumes a global minimum or converges to its infimum for $t \to \infty$. Which behaviour occurs can now be precisely described by the magnetic momentum.

\begin{prop} \label{charByMomentum}
For each $(q,p)\in U_E$ we have
\begin{enumerate}
	\item $I = \Itr\ \Longleftrightarrow \ \lim_{t\to\infty} \norm{q(t)} = R^+$.
	\item $I < \Itr\ \Longrightarrow \ \min_{t\in\bR} \norm{q(t)}> R^+$.
\end{enumerate}
\end{prop}

\begin{proof}
As for the first assertion, we prove ``$\Rightarrow$'' by contradiction. If $\norm{q(t)}$ does not converge to $R^+$ then it attains its minimum $\rho > R^+$ and, since Lemma~\ref{thetaLeZero} will assure $\dot\theta < 0$, we have 
\begin{equation*}
\begin{split}
I &= -\rho\sqrt{2E} - \int_\rho^R B(r)r \intd{r} \\
 &= -\rho\sqrt{2E} + \int_{R^+}^\rho B(r)r \intd{r} - \int_{R^+}^R B(r)r \intd{r} \\
 &< -\rho\sqrt{2E} + (\rho - R^+)\sqrt{2E} - \int_{R^+}^R B(r)r \intd{r} \\
 &= -R^+\sqrt{2E} -  \int_{R^+}^R B(r)r \intd{r} \\
 &= \Itr,
\end{split}
\end{equation*}
which contradicts the premise. The reversed implication is true since the curvature near $\{ \norm{q} = R^+ \}$ is negative and hence a trajectory with $\norm{q(t)} \to R^+$ has to have negative values of $\dot\theta$. Then the trajectory converges to the circular orbit, giving $I = \Itr$.

For the second assertion, we have to show that $\norm{q(t)}$ cannot converge to its infimum as $t \to \infty$. If this was the case we would have $\sprod{q(t)}{p(t)} \to 0$. By compactness there is a sequence $t_n \to \infty$ of times such that $(q(t_n),p(t_n)) \to (q_\infty,p_\infty) \in \Sigma_E$ with $\sprod{q_\infty}{p_\infty} = 0$. By Prop.~\ref{VBed-Allg} its trajectory leaves the support which contradicts the continuity of the flow.
\end{proof}

\begin{ex}
We want to illustrate some of these results by an explicit example. Since the magnetic field $B=B(r)$ need only be locally Lipschitz, we may consider the simplest type of Lipschitz continuous function having compact support, e.g., $B:[0,\infty)\to\bR$ where
\begin{equation*}
B(r) = \begin{cases} 10(1-r) & \text{if } r\in[0,1] \\ 0 & \text{if } r\geq 1 \end{cases}
\end{equation*}
Then the magnetic momentum $I$ for the energy $E=1/2$ can be written as in \eqref{eq:integralPolarCoords}. The graph of $B$, together with the graph of the function $r\mapsto 1/r$, and the level sets of $I$ are shown in Fig.~\ref{fig:ex}.

\begin{figure}[ht]%
\begin{subfigure}[b]{.5\linewidth}
\centering
\includegraphics[width=0.8\columnwidth]{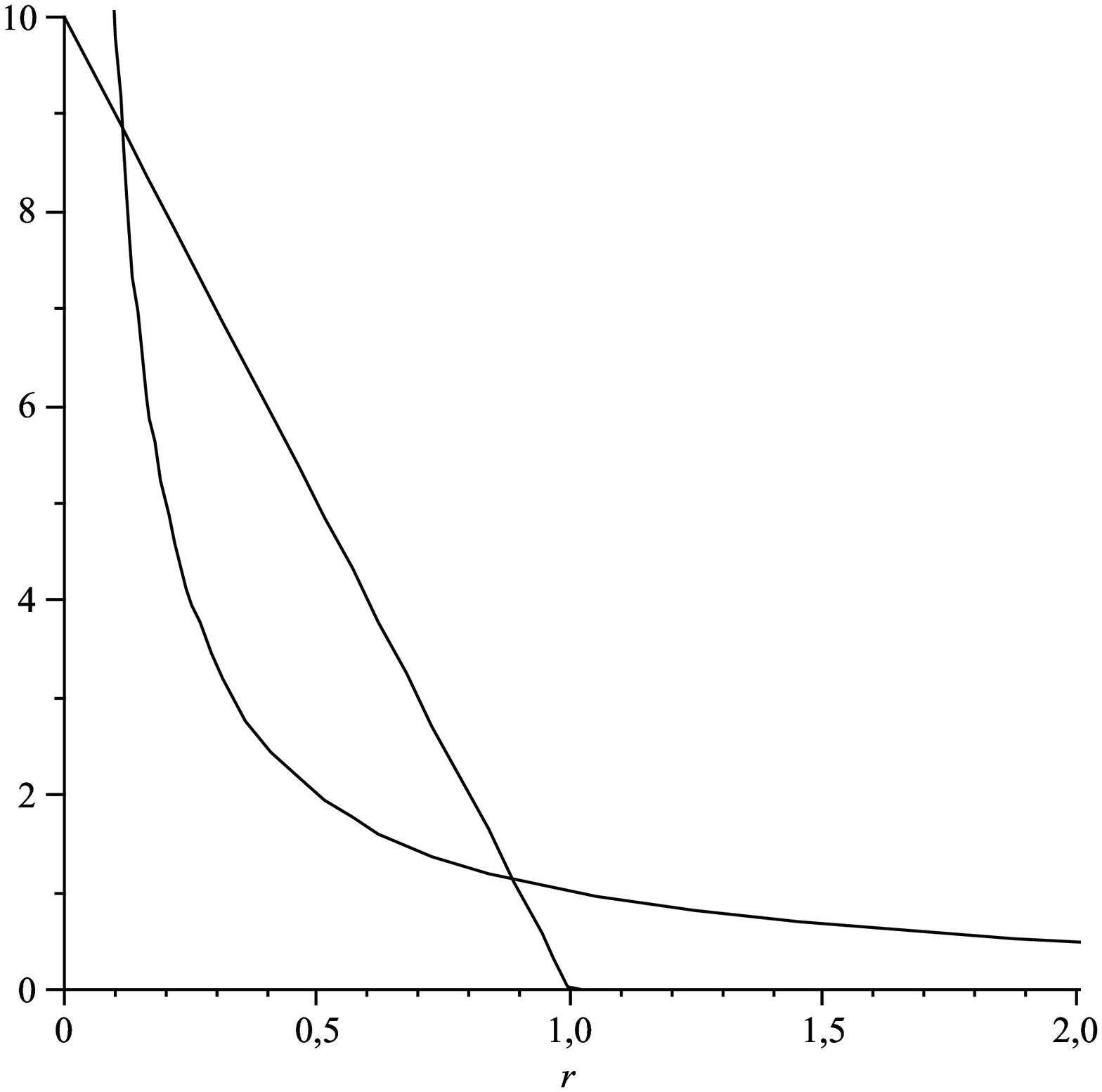}%
\caption{The graph of $B$}
\label{fig:exgraph}
\end{subfigure}
\begin{subfigure}[b]{.5\linewidth}
\centering
\includegraphics[width=0.8\columnwidth]{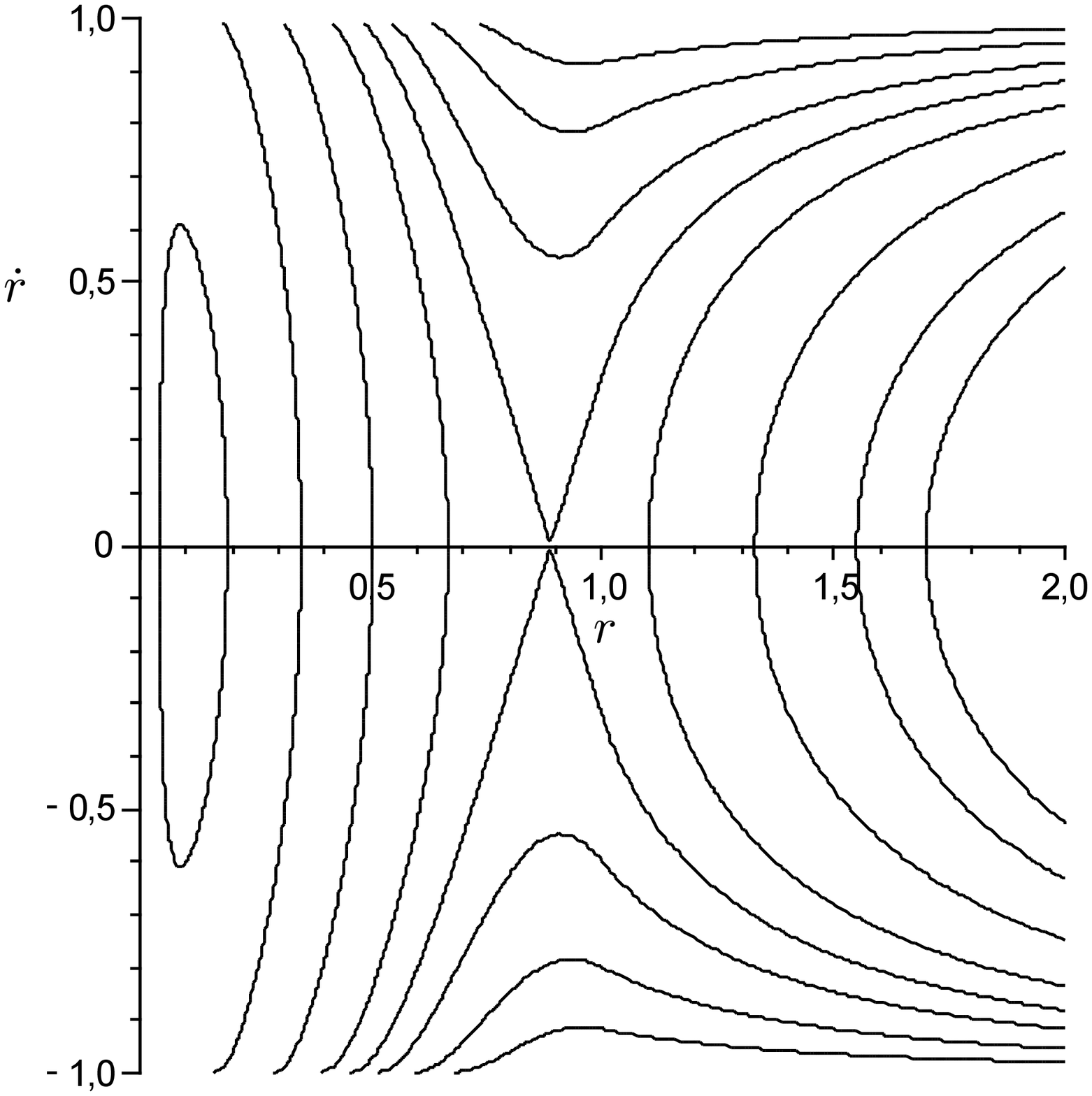}%
\caption{The level sets of $I$}
\label{fig:exlevelsets} 
\end{subfigure}
\caption{The magnetic field and its magnetic momentum}
\label{fig:ex}%
\end{figure}

Since there are only two intersections of the two graphs in Fig.~\ref{fig:exgraph}, we have precisely two circular orbits which are represented by the points $(R^{\pm},0)$ in Fig.~\ref{fig:exlevelsets}; the outermost circular orbit has radius $R^+ \approx .8873$. The level set of the critical momentum is the level set where $I=\Itr=I(R^+,0)\approx -.946$. One clearly recognizes in this example, e.g., the characterization of orbits having magnetic momentum $I\geq \Itr$ given in Prop.~\ref{charByMomentum}.
\end{ex}

For points $x \in U_E$ with $I(x) \le \Itr$ we now consider the angle 
\[
\theta(x,t) := \int_0^t \dot\theta(\tau) \intd{\tau} = \int_0^{t} \frac{\sprod{q(\tau)}{Jp(\tau)}}{\norm{q(\tau)}^2} \intd{\tau}
\]
as well as, for $I(x) < \Itr$, the exit angle $$ \theta^e(x) := \theta(T(x),x) $$ where $T(x)$ denotes the exit time of $x$. Then Lemma~\ref{thetaLeZero} immediately yields the next result.

\begin{prop} \label{thetaToInfty}
For fixed $x$ the function $\theta(x,t)$ is strictly decreasing with respect to $t$.
For $x$ with $I(x) = \Itr$ we have $\theta(x,t) \to -\infty$ as $t \to \infty$.
We have $\theta^e(x) \to -\infty$ as $I(x) \to \Itr$.
\end{prop}

\begin{rem} \label{integralRThetaGeneral}
The previous observations hold for $B(R^+) > 0$. In fact, these calculations work in a similar way for the case $B(R^+) < 0$ where one has to switch the sign of $\dot\theta$ since the circular orbit turns in the opposite direction. Thus, for $B(R^+) < 0$, we obtain
\begin{equation} \label{eq:DefTrMom+}
\Itr = R^+ \sqrt{2E} - \int_{R^+}^R B(r)r \intd{r}
\end{equation}
as the critical magnetic momentum, and the analogue of Prop.~\ref{charByMomentum} now reads
\begin{enumerate}
	\item $I = \Itr\ \Longleftrightarrow \ \lim_{t\to\infty} \norm{q(t)} = R^+$ (no changes here).
	\item $I > \Itr \ \Longrightarrow \ \min_{t\in\bR} \norm{q(t)}> R^+$.
\end{enumerate}
Furthermore, we have $\dot\theta > 0$ for $I \ge \Itr$, and $\theta^e(x) \to \infty$ as $I(x) \to \Itr$.
\end{rem}


\section{Symbolic Dynamics} \label{Sect:SymbDyn}

In the following we consider a magnetic field $$ B = \sum_{k=1}^n B_k $$ where each component $B_k$ is rotationally symmetric with respect to its respective center $q_k\in\bR^2$, i.e., $B_k= B_k(\|q-q_k\|)$, with  $$ \supp B_k\subseteq D_k := \{q\in\bR^2\mid \|q-q_k\|\le R_k\}
\qquad\mbox{and}\qquad D_k \cap D_l = \emptyset\quad \mbox{ for }k\neq l. $$
 For a rotationally symmetric magnetic field $B$, we saw in Section~\ref{Sect:Symm} that, for energies below some threshold $E^\circ = E^\circ(B)> 0$, there is a circular orbit. In the case of several components we require the existence of circular orbits for each component $B_k$ and hence define 
\[
E^\circ := \min_{1\leq k\leq n} E^\circ(B_k) > 0 .
\]

We point out that the magnetic momentum introduced in Prop.~\ref{integralCartesianCoords} is no longer a constant of motion since $B$ is not rotationally symmetric anymore, but for each $B_k$ we may compute its local magnetic momentum $I_k$ (with $q_k$ taking the part of the origin) and obtain a local integral  in the sense that $I_k$ is constant along a trajectory as long as it stays outside the other supports.

For the following considerations we fix an energy $E \in (0, E^\circ]$ and let $R_k^+ = R_k^+(E)$ denote the radius of the outermost circular orbit in $D_k$. We write $\Itr_k = \Itr_k(E)$ for the critical magnetic momentum for $B_k$ as in \eqref{eq:DefTrMom-} or \eqref{eq:DefTrMom+}, depending on the sign of $B_k$ along the circular orbit of radius $R_k^+$.

\begin{rem}
Instead of looking at the value $I(x)$ of the magnetic momentum we can consider the angle $\alpha$ between $q$ and $Jp$. Since we are interested in certain entrance directions it is more convenient for the following to think of a critical angle $\atr_k$ instead of a critical momentum $\Itr_k$. We have an orientation preserving homeomorphism between the momentum $I \in [- R \sqrt{2E}, R \sqrt{2E}]$ and the angle $\alpha \in [- \frac{\pi}{2}, \frac{\pi}{2}]$.
\end{rem}

In what follows we will need two geometric conditions on the configuration of the components $\supp B_k$ of the support of the magnetic field. Magnetic fields satisfying these conditions are said to be in general position, and we will derive a precise definition in the following (see Def.~\ref{defgenpos}). 
\begin{enumerate}
\item
The first condition will allow transitions from one support to any other. 

The easiest way to guarantee this would be to demand for the convex hull of two supports to have empty intersection with the other supports. 

However, this is a rather restrictive condition and we will use a weaker one, only demanding for certain parts of the convex hull---depending on the chosen energy---to have an empty intersection; see Fig.~\ref{fig:shadowingSupports} for an illustration. 
\item 
The second condition will assure that we can choose an appropriate Poincar\'e section which counts the revolutions around a given center $q_k$ in the right way.

\begin{figure}[ht]%
\centering
\includegraphics[width=0.4\columnwidth]{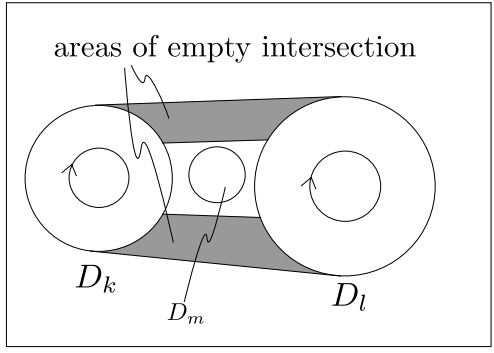}%
\caption{Areas of empty intersection}%
\label{fig:shadowingSupports}%
\end{figure}

The transition corridors between two supports, as depicted in Fig.~\ref{fig:shadowingSupports}, depend on the orientation of their circular orbits, or equivalenty, on the sign of $B$ along the orbit. For this we consider the four possible sign combinations. If for fixed $k \neq l$ both signs are positive we consider the tangent line to $\partial D_k$ and $\partial D_l$ passing both supports to the left and the line which hits $\partial D_k$ with angle $\atr_k$ and $\partial D_l$ with angle $\atr_l$; see Fig.~\ref{fig:DefA_kl_1}. 

By $A_{k,l}$ and $B_{l,k}$ we denote the part of the set of points on $\partial D_k$ and $\partial D_l$ between these two lines as shown in Figure \ref{fig:DefA_kl_2}. If both signs are negative we basically consider the same lines as in the previous case but with left tangent replaced by right tangent.

\begin{figure}[ht]%
\begin{subfigure}[b]{.5\linewidth}
\centering
\includegraphics[width=0.8\columnwidth]{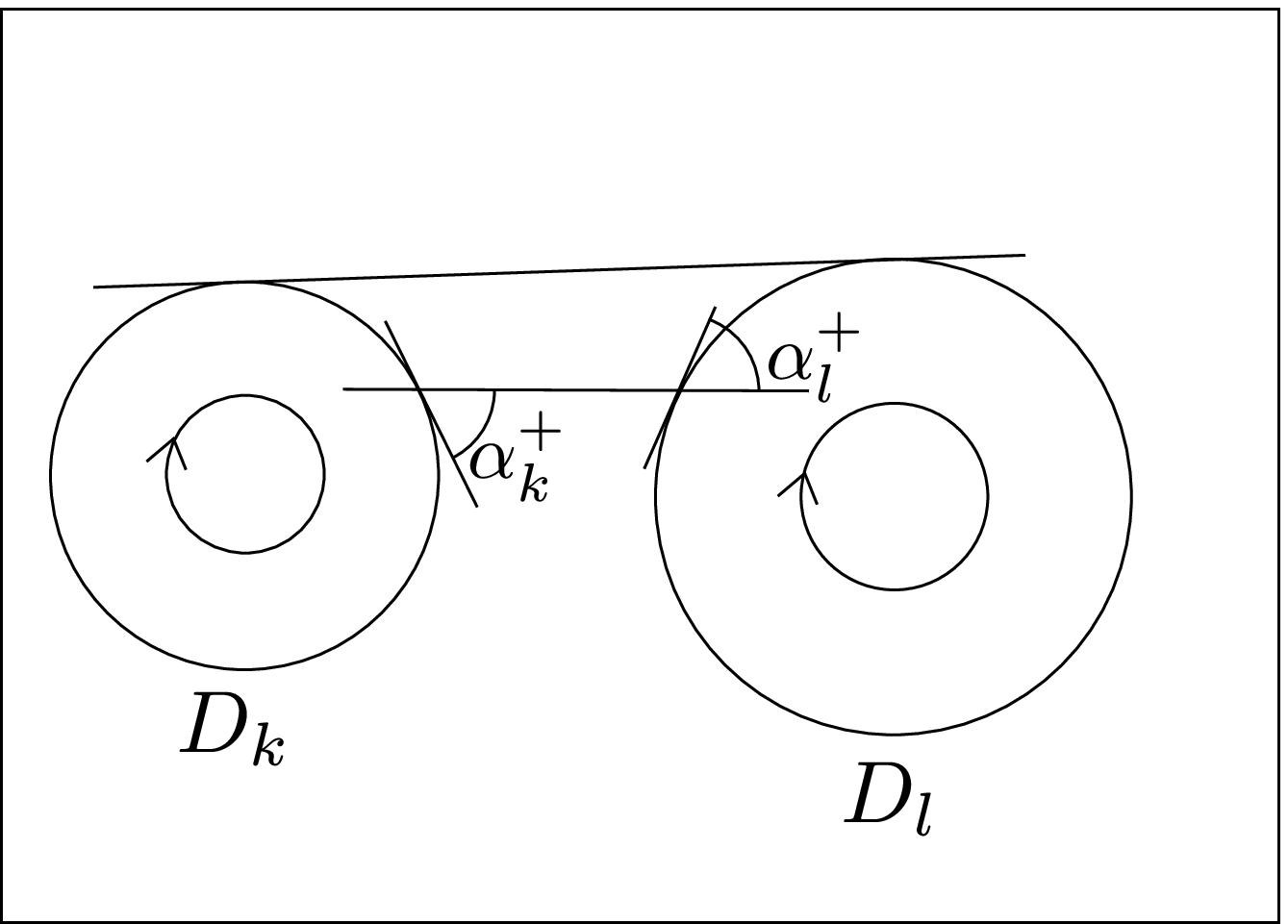}%
\caption{$\atr_k$ and $\atr_l$}
\label{fig:DefA_kl_1}
\end{subfigure}
\begin{subfigure}[b]{.5\linewidth}
\centering
\includegraphics[width=0.8\columnwidth]{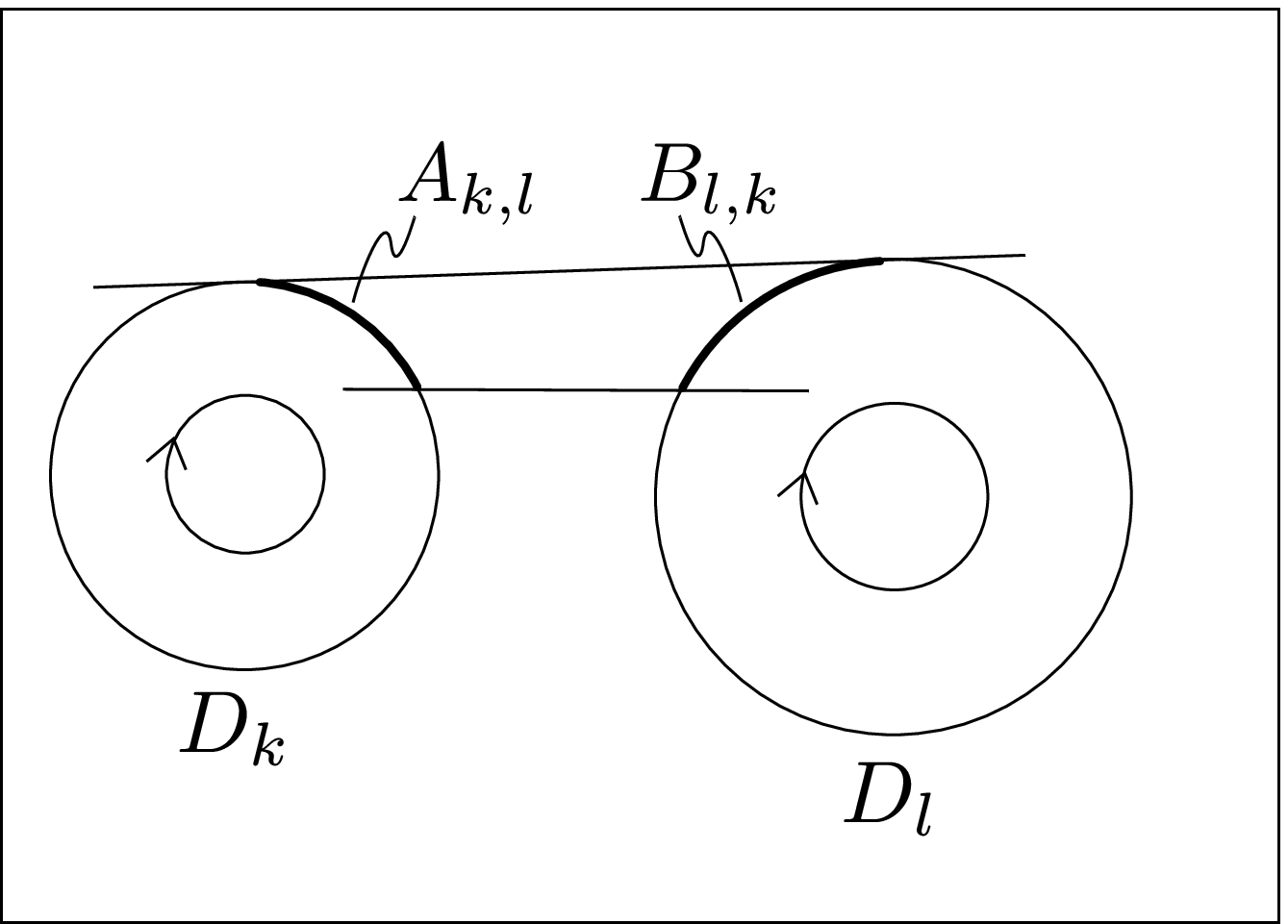}%
\caption{$A_{k,l}$ and $B_{l,k}$}
\label{fig:DefA_kl_2} 
\end{subfigure}
\caption{Definition of $\atr_k$ and $\atr_l$, respectively, $A_{k,l}$ and $B_{k,l}$}%
\label{fig:DefAB_kl}%
\end{figure}

\begin{figure}[ht]
\centering
\includegraphics[width=0.4\columnwidth]{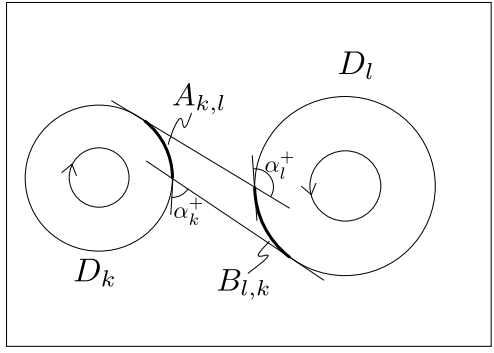}%
\caption{The case of opposite signs}
\label{fig:DefAklBkl+-} 
\end{figure}

If one sign is positive, say of $B_k$, and the other, say of $B_l$, is negative along the circular orbits we consider the left tangent to $D_k$ which hits $\partial D_l$ with angle $\atr_l$ and the line which hits $\partial D_k$ with angle $\atr_k$ and is tangent to $D_l$ on the right hand side; see Fig.~\ref{fig:DefAklBkl+-} for a visualisation and the choice of $A_{k,l}$ and $B_{l,k}$ in this case.
\end{enumerate}

Let us point out that the sets $A_{k,l}$ and $B_{l,k}$, as well as the critical angles $\atr_k$ and $\atr_l$, depend on the energy value $E\in (0,E^{\circ}]$. We now give the precise definition of the geometric configuration we are going to impose. We denote by $\conv M$ the convex hull of a set $M$.

\begin{defin} \label{defgenpos}
Let $B = \sum_{k=1}^n B_k$ be a magnetic field on $\bR^2$ such the $B_k$ are rotationally symmetric (with respect to their respective centers $q_k$) and have pairwise disjoint supporting
disks $D_k = \{ q\in\bR^2 \mid \norm{q - q_k} \le R_k\}$. 
We then say that $B$ is {\bf in general position with respect to an energy} 
$E \in (0,E^\circ]$ if we have
\begin{enumerate}
	\item $\conv (A_{k,l} \cup B_{l,k}) \cap \supp B_m = \emptyset$ for distinct $k, l, m \in \{1, \dots, n \}$, and
	\item $\partial D_k \setminus (\bigcup_{l \neq k} A_{k,l} \cup \bigcup_{l \neq k} B_{k,l}) \neq \emptyset$ for all $k$.
\end{enumerate}
\end{defin}

\begin{rem}
\begin{enumerate}
  \item Note that the $B_k$ need not have a fixed sign. There may also be parts of the disc $\{\norm{q-q_k} \le R_k \}$ where $B_k$ vanishes.
  \item The first condition guarantees that all the transitions from $D_k$ to any other $D_l$ are possible. The second condition allows us to put a Poincar\'e section at an appropriate place; it will be a radial segment ending in the nonempty set considered there.
  \item Replacing the first condition by the much weaker condition $\conv (D_k \cup D_l) \cap D_m = \emptyset$ would still guarantee all the transitions, but exclude very simple configurations ($n=4$ is enough) which then would not satisfy Condition~2.
\end{enumerate}
\end{rem}

Condition~2 assures the existence of points 
$$ q_k^* \ \in \ \partial D_k \Big\backslash \Big(\bigcup_{l \neq k} A_{k,l} \cup \bigcup_{l \neq k} B_{k,l}\Big) , $$ 
and we pick such points $q_k^*$ and hold them fixed once and for all. With these points we now define the Poincaré sections 
\[
P_{k,E} = \big\{ (q,p) \in \Sigma_E \sd q = q_k + \lambda (q_k^* - q_k) \text{ for some }\lambda \in (0,1), \sprod{q-q_k}{Jp} \neq 0 \big\} ,
\]
as well as the Poincaré map 
$$ p:\ P_E := \bigcup_{k=1}^n P_{k,E} \ \longrightarrow  \ \overline{P_E} \cup \{ \infty \} $$ 
by defining $p(x)$ as the first point $\Phi(t,x) \in \overline{P_E}$ for $t > 0$ if such a point exists; we set $p(x) = \infty$ if such a point does not exist. See Fig.~\ref{fig:DefPoincareMap}) for illustration. 
Then $p$ is continuous at $x\in P_{E}$ if $p(x) \in P_{E}$. In particular, if $x_m \to x_\infty \in P_E$ and $p(x_m) \to y_\infty \in P_E$, then the continuity of the flow yields $p(x_\infty) = y_\infty$ and $p$ is continuous at $x_\infty$.

\begin{figure}[ht]%
\centering
\includegraphics[width=0.4\columnwidth]{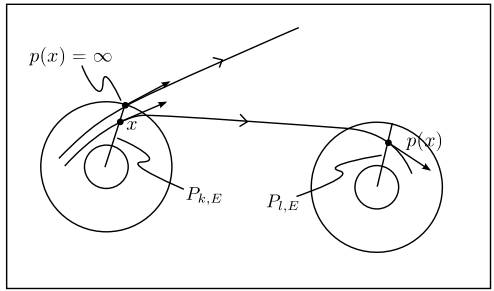}%
\caption{Definition of the Poincaré map $p$}%
\label{fig:DefPoincareMap}%
\end{figure}

Then 
\begin{equation} \label{LambdaE}
 \Lambda_E = \{ x\in P_E \mid p^i(x) \in P_E \text{ for all } i \in \bZ \} 
\end{equation}
is an invariant set for $p$ which, at this point, might still be empty. Let 
\begin{equation} \label{semi}
h: \Lambda_E \to S = \{ 1, \dots, n\}^\bZ
\end{equation}
denote the canonical coding map where $h(x) = (s_i)$ such that $p^i(x) \in P_{s_i}$ for all 
$i \in \bZ$, and let $\sigma: S \to S$ be the left shift map, shifting a sequence $(s_i)$ 
one position to the left. By construction we get the identity 
\[ h \circ p = \sigma \circ h . \]

We now formulate the main theorem of this paper.

\begin{thm} \label{thmSymDyn}
If the planar magnetic field $B$ is in general position with respect to the energy $E\in (0,E^\circ]$, 
then there is a nonempty $p$-invariant subset $\Lambda_E^{'} \subset \Lambda_E$ 
such that $h: \Lambda_E^{'} \to S$ is continuous and surjective; 
in other words, $p: \Lambda_E^{'} \to \Lambda_E^{'}$ is semi-conjugated to the full 
shift $\sigma: S \to S$.
\end{thm}

Since a dynamical system has at least the entropy of the one it is semi-conjugated to, 
we immediately obtain the following corollary for the corresponding magnetic flow. Here we use Bowen's definition of topological entropy (see, e.g., \cite[\S~7.2]{Wa}), and denote the flow restricted to the energy surface $\Sigma_E$ by $\Phi_E := \Phi|_{\bR\times \Sigma_E}$.

%
%
\begin{cor}
For $n\ge2$ bumps and a planar magnetic field $B=\sum_{k=1}^n B_k$ in general position, there is a constant $c>0$ so that, for each $E\in (0,E^\circ]$, the topological entropy is bounded from below by
$$ h_{top}(\Phi_E) \ge c \sqrt{E} . $$
\end{cor}
\begin{proof}
By the monotonicity of the topological entropy with respect to the inclusion of invariant sets \cite[Prop.~3.1.7(1)]{KH}, it suffices to show that the topological entropy of the flow
$\Phi_{B_E}:= \Phi|_{\bR\times B_E}$ satisfies the same lower bound,
with $B_E:=\Phi(\bR,\Lambda_E)$ and $\Lambda_E\subseteq P_E$
defined in (\ref{LambdaE}).

We now use Thm.\ \ref{thmSymDyn} above. By the decrease of the topological entropy with respect to semi-conjugacies \cite[Thm.~7.2]{Wa}, $h_{top}(\Phi_{B_E})$ is at least the topological entropy of the special flow $\psi$ given in \cite[\S 0.3]{KH} of the full shift 
with a certain roof function $T_E:S\to (0,\infty)$. Indeed, as a surjection, the semi-conjugacy $h$ from (\ref{semi}) has a left inverse $S\to \Lambda'_E$, and choosing such an inverse, $T_E$ is defined as its pull-back of the Poincar\'e return time
$$ \tilde T_E:\Lambda'_E\to (0,\infty),\qquad  \tilde T_E(x):=\inf\{t>0\mid \Phi(t,x)\in P_E\} . $$

We claim that there is a constant $c'>0$ such that, for each $E\in(0,E^\circ)$, $\tilde T_E$ is uniformly bounded above by $c'/ \sqrt{E}$. This is true for points of $\Lambda'_E$ being mapped to a Poincar\'e surface  belonging to a different bump, since on $\Sigma_E$ the speed equals $\sqrt{2E}$ and the mutual distances of the configuration space projections of the Poincar\'e surfaces are uniformly bounded from above. Then the Poincar\'e time is bounded from above by that distance bound, divided by the speed. For points of $\Lambda'_E$ being mapped to the Poincar\'e surface of the same bump, a similar argument
applies, since the total length of the trajectory between successive intersections of Poincar\'e surfaces is bounded from above by a constant times the circumference of the disk $D_k$.

By applying Abramov's formula, see Thm.~2.2.1 in \cite[Ch.~3]{Si}, we finally conclude like in the proof of Thm.~14.4 of \cite{kn2} that $$ h_{top}(\psi)\ge \frac{\log n}{c'} \sqrt{E} . $$ This implies the assertion of the corollary with $c:=\frac{\log n}{c'}$.
\end{proof}


\subsection{Proof of Theorem~\ref{thmSymDyn}}

For the following considerations we need to define various sets and maps. Recall from 
(\ref{eq:UE1}) in Sect.~\ref{Sect:Symm} the definition of the set $U_E$ of points in 
phase space through which $q$-space orbits can enter the support; due to the different 
centers of the supports this will now replaced by 
\[
U_E = \bigcup_{k=1}^n U_{k,E}\quad\mbox{ where }\quad
U_{k,E} = \big\{ (q,p) \in \Sigma_E \sd \norm{q - q_k} = R_k,\ \sprod{q-q_k}{p} \le 0 \big\} .
\]
We also need two additional maps that describe the relation between $P_E$ and $U_E$. 
 For some orbit entering the support let \[u: U_E \to P_E \cup \{ \infty \}\] denote the first point 
 $\Phi(t,x) \in P_E$, $t \ge 0$, to hit the Poincaré section. For some orbit intersecting $P_E$ 
 let \[v: P_E \to U_E \cup \{ \infty \}\] denote the first point $\Phi(t,x) \in U_E$, $t \ge 0$, 
 where the orbit enters the support again. In case such points do not exist $u(x)$, 
 respectively $v(x)$, are set to $\infty$.

In order to understand the basic mechanism to find a point that realises a prescribed itinerary, 
we start with a segment $\Gamma\subset U_k$ consisting of phase space points 
entering $D_k$.
We assume that the trajectory of 
one endpoint of $\Gamma$ passes $D_k$ tangentially to the left and the other 
endpoint has critical momentum 
$\Itr_k$, i.e., its trajectory converges to the outermost circular orbit inside $\supp B_k$,
see Fig.~\ref{fig:shootingMechanismA}. 

Then we find in $\Gamma$ a subsegment of points whose trajectories hit 
$P_k$ exactly $j$ times before leaving $D_k$ towards $D_l$. 
In particular, the resulting segment of points in $U_l$ has the same configuration 
as the original segment entering $U_k$: the trajectory of one endpoint passes 
$D_l$ tangentially to the left and the other endpoint has critical momentum $\Itr_l$. 
This behaviour is illustrated in Fig.~\ref{fig:shootingMechanismB} and made precise 
in the following lemma.

\begin{figure}[ht]%
\begin{subfigure}[b]{.5\linewidth}
\centering
\includegraphics[width=.8\columnwidth]{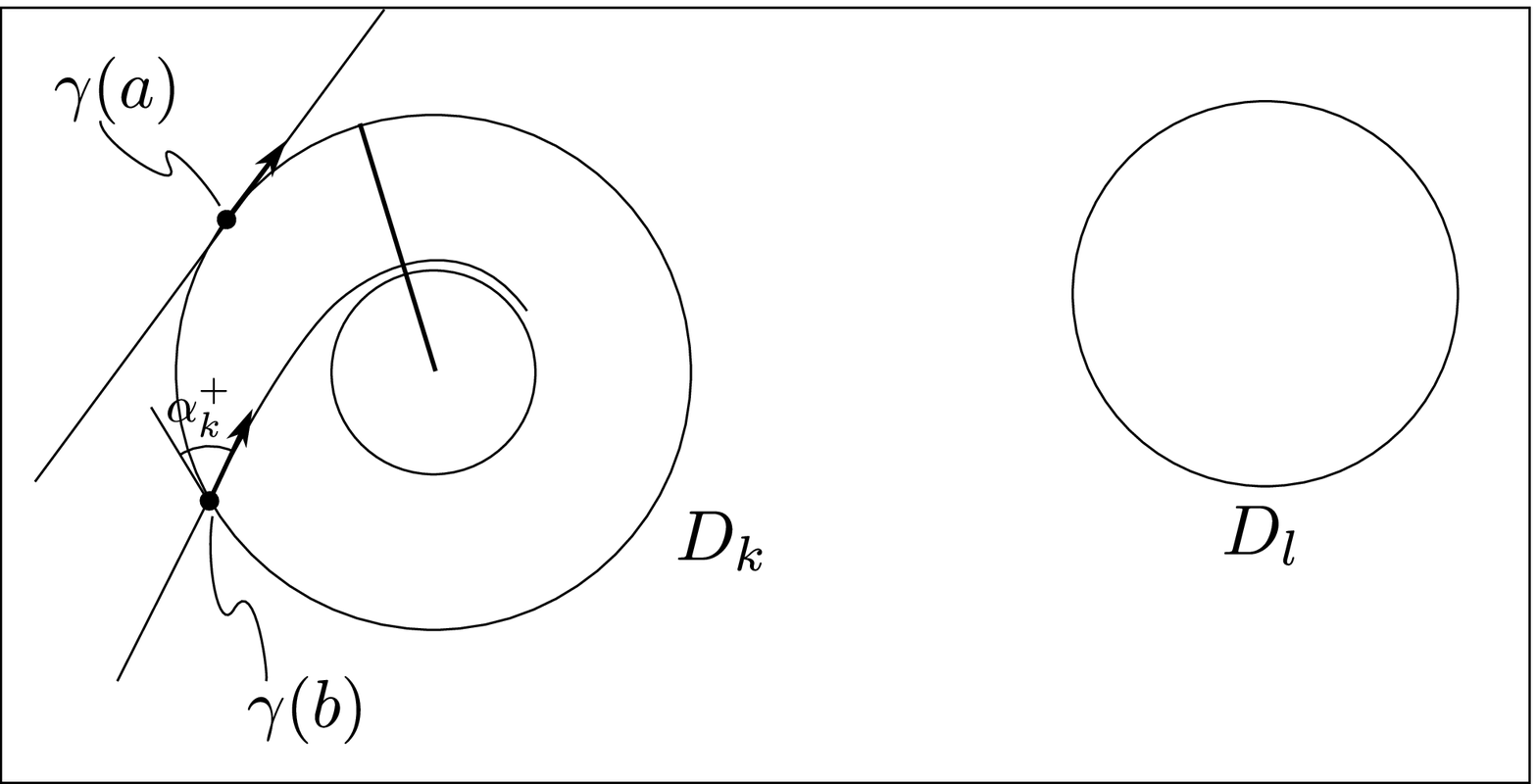}%
\caption{Endpoints of original set}%
\label{fig:shootingMechanismA}
\end{subfigure}
\begin{subfigure}[b]{.5\linewidth}
\centering
\includegraphics[width=.8\columnwidth]{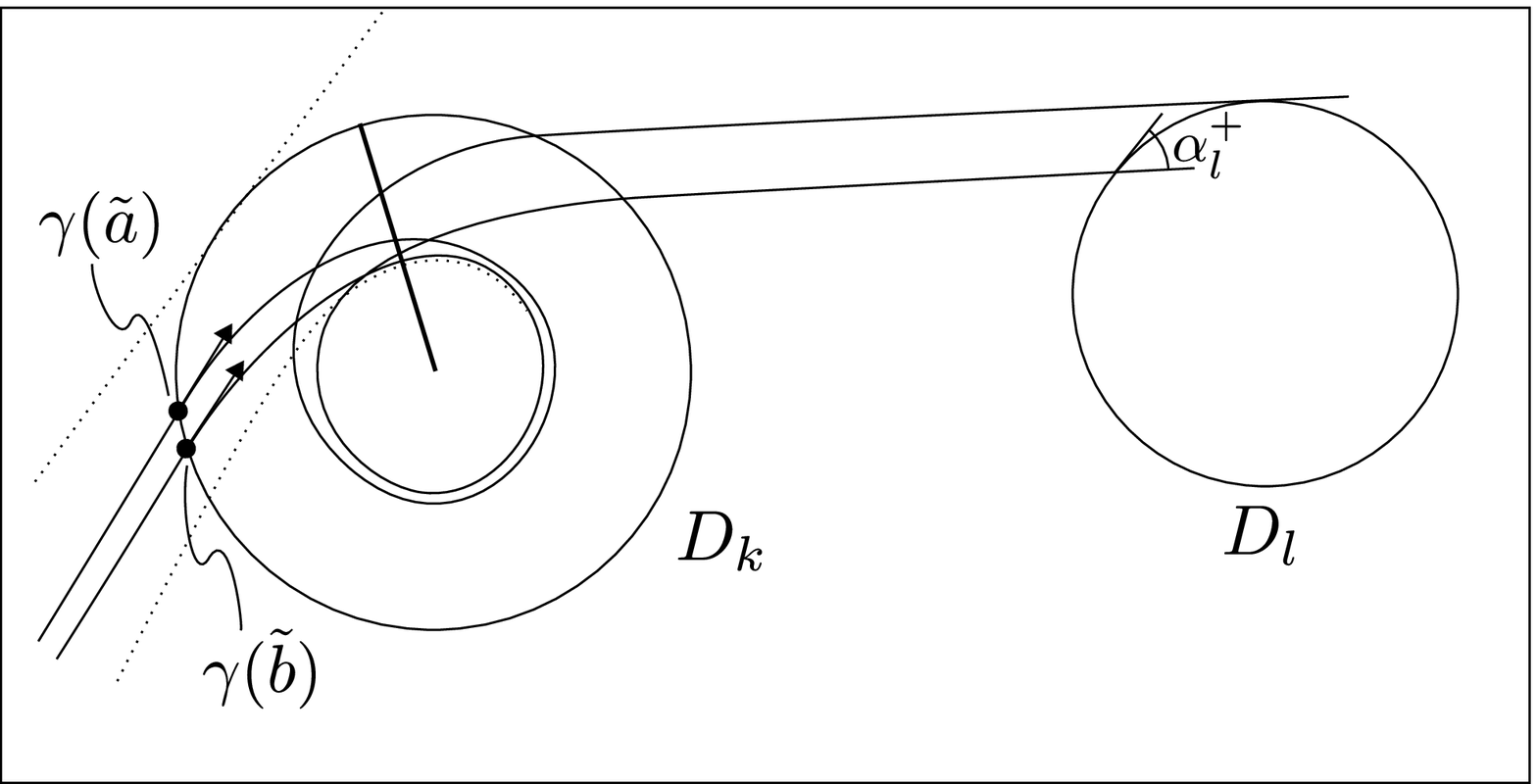}%
\caption{Endpoints of resulting set}%
\label{fig:shootingMechanismB}
\end{subfigure}
\caption{Shooting mechanism for $j = 2$}
\label{fig:shootingMechanism}%
\end{figure}

\begin{lem} \label{shootingMechanismSpecial}
Let the magnetic field $B$ be in general position with respect to the energy $E$. 
Let $k \neq l \in \{1, \dots, n\}$ be different indices of supports and $j \in\bN$. 
Let $\gamma : [a,b] \to U_k$ with $I_k\big(\gamma(a)\big) = -R_k \sqrt{2E}$ and 
$I_k\big(\gamma(b)\big) = \Itr_k$.

Then there exists $[\tilde a, \tilde b] \subset [a,b]$ such that for all points 
$x \in \im \gamma|_{[\tilde a, \tilde b]}$ in the image of $\gamma|_{[\tilde a, \tilde b]}$ we have
\begin{enumerate}
	\item $p^i \big( u(x) \big) \in P_k$ for $0 \le i \le j-1$
	\item $p^j\big( u(x) \big) \notin P_k$
	\item $v \big( p^j(u(x)) \big) \in U_l$
\end{enumerate}
and the curve 
$\gamma_1 = \left.(v \circ p^j \circ u \circ \gamma)\right|_{[\tilde a, \tilde b]}: 
[\tilde a, \tilde b] \to U_l$ 
satisfies $I_l\big(\gamma_1(\tilde a)\big) = -R_l \sqrt{2E}$ and 
$I_l\big(\gamma_1(\tilde b)\big) = I_l^+$.
\end{lem}

This description is valid for the case of the magnetic fields being positive along the 
outermost circular orbits, and the statement (as well as the proof) are also given for this case. 
The general situation is then treated in Rem.~\ref{shootingMechanismGeneral}.

\begin{proof}
Let $q_a, q_b \in \partial D_k$ denote the two endpoints of $\partial A_{k,l}$ 
as shown in Fig.~\ref{fig:q_ab}.

\begin{figure}[ht]%
\centering
\includegraphics[width=0.4\columnwidth]{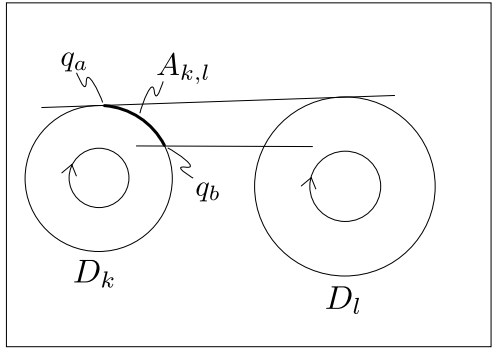}%
\caption{Definition of $q_a$ and $q_b$}%
\label{fig:q_ab}%
\end{figure}

Since $I_k\big(\gamma(s)\big) \to \Itr_k$ for $s\to b$ we have 
$\theta_k(\gamma(s)) \to -\infty$ as $s \to b$ by Lemma~\ref{thetaToInfty}. 
Hence there are parameters $s_a < s_b$ in $[a,b]$ such that the trajectory of 
$\gamma(s_a)$ exits through $q_a$, the trajectory of $\gamma(s_b)$ exits through 
$q_b$, and for $s_a \le s \le s_b$ the forward trajectory hits $P_k$ exactly $j$ 
times before leaving $D_k$. The trajectory of $\gamma(s_a)$ passes $D_l$ 
on the left hand side, and by Cor.~\ref{enteringExitingAngleEqual} the trajectory of 
$\gamma(s_b)$ passes to the right of the line hitting $D_l$ with angle $\atr_l$, 
hence it hits with some angle bigger than $\atr_l$, i.e., its magnetic momentum is bigger 
than $\Itr_l$.

Therefore there are new parameters $\tilde a < \tilde b$ in $[s_a,s_b]$ such that the trajectory of $\gamma(\tilde a)$ passes tangentially to the left of $D_l$, the trajectory of $\gamma(\tilde b)$ hits $D_l$ with angle equal to $\atr_l$, i.e., in a point with momentum equal to $\Itr_l$, and for $s \in [\tilde a, \tilde b]$ the trajectory of $\gamma(s)$ hits $D_l$.

Then $\gamma_1: [\tilde a, \tilde b] \to U_l$ with $\gamma_1(s) = (v \circ p^j \circ u \circ \gamma)(s)$ is well defined and possesses all the properties stated above.
\end{proof}

\begin{rem}\label{shootingMechanismGeneral}
The mechanism in the proof of Lemma~\ref{shootingMechanismSpecial} is manufactured for the situation where both magnetic fields have postive values on their outermost circular orbits. We argue that an analogous result holds for all other cases as well.

Indeed, in the case when the two magnetic fields have different signs on their outermost circular orbits, e.g., when $B_k>0$ and $B_l<0$, we can choose the curve $\gamma_1$ in such a way that the trajectory of $\gamma(\tilde a)$ passes tangentially to the right of $D_l$, i.e., $I_l(\gamma_1(\tilde a)) = R_l \sqrt{2E}$, and the trajectory of $\gamma(\tilde b)$ hits $D_l$ with angle $\atr_l$, i.e., $I_l(\gamma_1(\tilde b)) = \Itr_l$.

Finally, an analogous mechanism works when both $B_k$ and $B_l$ take negative values on their outermost circular orbits.
\end{rem}

The mechanism described above allows us to show that for every half-infinite sequence there is a point realising this prescribed itinerary.

\begin{prop} \label{halfInfiniteSequences}
For any half-infinite sequence $(k_i) \in \{1, \dots, n\}^{\bN_0}$ there exists $x_0 \in P_{k_0}$ such that $p^i (x_0) \in P_{k_i}$ for each $i \ge 0$.
\end{prop}

\begin{proof}
We start with a curve $\gamma: [0,1] \to U_{k_0}$ such that $I_{k_0}(\gamma(0)) = -R_{k_0} \sqrt{2E}$ (respectively $R_{k_0}\sqrt{2E}$ if $B_{k_0}$ is negative along the orbit) and $I_{k_0}(\gamma(1)) = \Itr_{k_0}$.

Let us assume first that $(k_i)$ will not become constant eventually. Let $j_1$ denote the number of consecutive values of $k_0$, i.e., $k_i = k_0$ for $i \le j_1 - 1$ and $k_{j_1} \neq k_0$. Then, by Lemma~\ref{shootingMechanismSpecial}, there are a parameter interval $[a^{(1)}, b^{(1)}] \subset [0,1]$ such that $p^i(u(\gamma(s))) \in P_{k_0} = P_{k_i}$ for $0 \le i \le j_1 - 1$ and some corresponding curve $\gamma_1: [a^{(1)}, b^{(1)}] \to U_{k_{j_1}}$. Applying Lemma~\ref{shootingMechanismSpecial} to the curve $\gamma_1$ and the number $j_2$ of consecutive values of $k_{j_1}$, we get $[a^{(2)}, b^{(2)}] \subset [a^{(1)}, b^{(1)}]$ such that $p^i(u(\gamma(s))) \in P_{k_{j_1}} = P_{k_i}$ for all $s \in [a^{(2)}, b^{(2)}]$ and $j_1 \le i \le j_1 + j_2 - 1$.

Hence, by iteration, we get $[a^{(m)}, b^{(m)}] \subset [a^{(m-1)}, b^{(m-1)}]$ such that $p^i (u(\gamma(s))) \in P_{k_i}$ for $i$ (at least) up to $m - 1$. Then there is a point $\bar x \in \bigcap_{m \ge 0} \gamma( [a^{(m)}, b^{(m)}] ) \neq \emptyset$ and hence $x_0 = u(\bar x) \in P_{k_0}$ satisfies $p^i (x_0) \in P_{k_i}$ for $i \ge 0$, which proves the claim if $(k_i)$ does not become constant.

If $(k_i)$ becomes constant eventually, we follow the same procedure but at the index where $(k_i)$ becomes constant, we simply choose the point with critical momentum and need not iterate further.
\end{proof}

In order to show that every bi-infinite sequence can be realised by some trajectory, we shift the sequence and show that the points obtained by Prop.~\ref{halfInfiniteSequences} converge to a point which has the prescribed itinerary. Let us define 
$$ \Lambda_E^{'} := \big\{ x\in \Lambda_E \mid 
\norm{q(t,x) - q_k} \ge R_k^+\ \forall\ t \in \bR,\ k = 1, \dots, n\} \big\}$$ 
as the subset of points whose trajectories stay outside the open disks 
$\{ \norm{q-q_k} < R_k^+ \}$ for all times.

In a first step, we need to show that $h|_{\Lambda_E^{'}}$ is surjective. Let $(k_i) \in \{1, \dots, n \}^\bZ$ be given. For each $m \in \bN$, Prop.~\ref{halfInfiniteSequences} yields a point $y_m \in P_{k_{-m}}$ such that $p^i(y_m) \in P_{k_{i - m}}$ for $i \ge 0$, and we define $$ x_m := p^m(y_m) \in P_{k_0} . $$ Since $\overline{P_{k_0}}$ is compact we have a convergent subsequence, which we denote by $x_m$ again, i.e., we have $x_m \to x_\infty \in \overline{P_{k_0}}$ as $m \to \infty$. We claim that $x_\infty$ lies in $\Lambda_E^{'}$ and satisfies $h(x_\infty) = (k_i)$, i.e., $$ p^i(x_\infty) \in P_{k_i} $$ for all $i\in\bZ$.

In order to prove that $x_{\infty}\in \Lambda_E^{'}$ we first claim that the trajectory of $x_\infty$ does not intersect $\partial P_E$. Arguing by contradiction, we assume that this happens at the index $i=0$, say, so that we have
\[ x_\infty = (q_\infty, p_\infty) \in \partial P_{k_0,E} = \{ (q,p) \in \Sigma_E \sd q = q_{k_0} \text{ or } q = q_{k_0}^* \text{ or } \sprod{q-q_{k_0}}{Jp} = 0 \} . \]
Now, $\sprod{q_\infty - q_{k_0}}{Jp_\infty} = 0$ is not possible since $\sprod{q - q_{k_0}}{Jp} = \norm{q-q_{k_0}}^2 \dot\theta < 0$ by Lemma~\ref{thetaLeZero}. This also implies that $q=q_{k_0}$ cannot occur. Therefore, we must have $q_\infty = q_{k_0}^*$ where three cases can occur, each of which will lead to a contradiction.

\emph{Case 1}: $k_1 \neq k_0$. Then the $q$-space trajectories $q(t,x_m)$ hit $D_{k_1}$, and by continuity of $q$ and closedness of the support, so does $q(t,x_\infty)$. But this contradicts the choice of $q_{k_0}^* \notin A_{k_0,k_1}$, as shown in Fig.~\ref{fig:ContradictionA-kl}.

\begin{figure}[ht]%
\centering
\includegraphics[width=0.4\columnwidth]{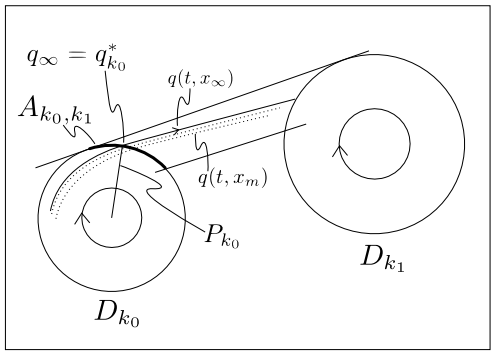}%
\caption{Contradiction in case 1}%
\label{fig:ContradictionA-kl}%
\end{figure}

\emph{Case 2}: $k_{-1} \neq k_0$. Similar to Case 1, the backward trajectory of $x_\infty$ in $q$-space would hit $D_{k_{-1}}$, contradicting the choice of $q_{k_0}^* \notin B_{k_0, k_{-1}}$.

\emph{Case 3}: $k_{-1} = k_0 = k_1$. Then $\norm{q(t,x_m) - q_{k_0}} < R_{k_0}$ for $|t| < t_0$ with $t_0 > 0$ independent of $m$. Hence $\norm{q(t,x_\infty) - q_{k_0}} \le R_{k_0} = \norm{q_\infty - q_{k_0}}$ for $|t| < t_0$. Then $\sprod{q_\infty - q_{k_0}}{p_\infty} = \left.\ddt \frac{1}{2} \norm{q(t,x_\infty) - q_{k_0}}^2\right|_{t=0} = 0$, i.e., the trajectory passes tangentially to $D_{k_0}$ and satisfies $\norm{q(t,x_\infty)- q_{k_0}} > R_{k_0}$ for small $|t|$. But this contradicts $\norm{q(t,x_m) - q_{k_0}} < R_{k_0}$.

Therefore the assumption $x_\infty \in \partial P_{k_0}$ is false and, hence, the trajectory of $x_\infty$ does not intersect $\partial P_E$ at all. By continuity we finally get \[ p^i(x_\infty) = \lim_{m \to \infty} p^i(x_m) \in P_{k_i} \] for all $i \in \bZ$. This finishes the proof that $h|_{\Lambda_E^{'}}$ is surjective.

It remains to show that $h|\Lambda_E^{'}$ is continuous. For this it is sufficient to show that the time between two consecutive intersections of the Poincaré section is uniformly bounded. Here, two cases are possible, namely the intersections can occur in the same support, or in different supports. Lemma~\ref{thetaLeZero} gives a bound for the time how long a trajectory can stay inside the outer annulus in one support. Furthermore, outside the support the trajectories are straight lines and, since $\norm{\dot{q}} = \sqrt{2E}$, the time is proportional to the length of the trajectory. Therefore, we have a uniform bound for the length of the segments between two consecutive intersections of the orbit with the Poincaré section and a uniform bound for the time. This proves the continuity of $h$ on $\Lambda_E^{'}$. 

Thus we have found a $p$-invariant set $\Lambda_E^{'} \neq\emptyset$ on which $h$ is continuous and surjective. This finishes the proof of our main result, Theorem~\ref{thmSymDyn}.


\section{Conclusion and further remarks} \label{Sect:Conclusion}

We have shown that a planar magnetic field with disjoint localized rotationally symmetric 
peaks creates positive topological entropy. 
To the best of our knowledge, this is the first instance where such systems have 
been investigated. In some sense, the classical potential analogue of our magnetic 
system should be the motion of a charged particle in a potential field with several Coulomb-type singularities. 
This system has been extensively studied in the literature, and it is well known that it carries positive topological entropy.

We point out that the existence of a second integral, although we made some advantage of it in our proofs, is not really necessary. Indeed, we only need the rotational symmetry in the region outside the outermost circular orbit. Inside the discs of radii $R_k^+$, the magnetic field may behave as wildly as it wants; this will not affect any of our constructions. Actually, we expect that even outside the discs of radii $R_k^+$ we do not need any symmetry. The mere existence of a (hyperbolic) outermost closed orbit bounding a convex region in the $q$-plane should be enough to detect symbolic dynamics. This will be the topic of future research.

Finally, the question arises whether there are higher-dimensional analogues of a magnetic field with localized peaks. At the moment, however, it is not clear what an appropriate generalization should be. Again, this will be investigated in the near future.

\end{document}